\documentclass{amsart}
\usepackage{amssymb}
\usepackage{amsmath,amsthm,amscd,amssymb, mathrsfs}
\usepackage{amsfonts}
\usepackage{latexsym}
\usepackage{color}

\usepackage{footmisc}

\usepackage[colorlinks,citecolor=red,pagebackref,hypertexnames=false]{hyperref}

\numberwithin{equation}{section}

\theoremstyle{plain}
\newtheorem{theorem}{Theorem}[section]
\newtheorem{lemma}[theorem]{Lemma}
\newtheorem{corollary}[theorem]{Corollary}

\newtheorem{conjecture}[theorem]{Conjecture}
\newtheorem{question}[theorem]{Question}

\newtheorem{claim}[theorem]{Claim}

\theoremstyle{definition}

\newtheorem{problem}[theorem]{Problem}

\theoremstyle{remark}

\newtheorem{case[theorem]}{Case}

\title[The generalized $k$-resultant modulus set]{The generalized $k$-resultant modulus set problem in finite fields}
\author{David Covert, Doowon Koh, and Youngjin Pi}

\address{Department of Mathematics and Computer Science\\
University of Missouri-St. Louis \\
St. Louis, MO 63121 USA} \email{covertdj@umsl.edu}

\address{Department of Mathematics\\
Chungbuk National University \\
Cheongju Chungbuk 28644, Korea}
\email{koh131@chungbuk.ac.kr}

\address{Department of Mathematics\\
Chungbuk National University \\
Cheongju Chungbuk 28644, Korea}
\email{pi@chungbuk.ac.kr}

\thanks{
Doowon Koh was supported by Basic Science Research Program through the National Research Foundation of Korea(NRF) funded by the Ministry of Education, Science and Technology(NRF-2015R1A1A1A05001374)}


\begin{document}

\begin{abstract} Let $\mathbb F_q^d$ be the $d$-dimensional vector space over the finite field $\mathbb F_q$ with $q$ elements.
Given $k$ sets $E_j\subset \mathbb F_q^d$ for $j=1,2,\ldots, k$, the generalized $k$-resultant modulus set, denoted by $\Delta_k(E_1,E_2, \ldots, E_k)$, is defined by
$$ \Delta_k(E_1,E_2, \ldots, E_k)=\left\{\|{\bf x}^1+{\bf x}^2+\cdots+{\bf x}^k\|\in \mathbb F_q:{\bf x}^j\in E_j,\, j=1,2,\ldots, k\right\},$$
where $\|{\bf y}\|={\bf y}_1^2+ \cdots + {\bf y}_d^2$ for ${\bf y}=({\bf y}_1, \ldots, {\bf y}_d)\in \mathbb F_q^d.$
We prove that if $\prod\limits_{j=1}^3 |E_j| \ge C q^{3\left(\frac{d+1}{2} -\frac{1}{6d+2}\right)}$ for $d=4,6$  with a sufficiently large constant $C>0$, then
$|\Delta_3(E_1,E_2,E_3)|\ge cq$ for some constant $0<c\le 1,$ and if $\prod\limits_{j=1}^4 |E_j| \ge C q^{4\left(\frac{d+1}{2} -\frac{1}{6d+2}\right)}$ for even $d\ge 8,$ then  $|\Delta_4(E_1,E_2,E_3, E_4)|\ge cq.$ This generalizes the previous result in \cite{CKP16}. We also show that if $\prod\limits_{j=1}^3 |E_j| \ge C q^{3\left(\frac{d+1}{2} -\frac{1}{9d-18}\right)}$ for even $d\ge 8,$ then  $|\Delta_3(E_1,E_2,E_3)|\ge cq.$ This result improves the previous work in \cite{CKP16} by removing $\varepsilon>0$ from the exponent.
\end{abstract}
\maketitle

\section{Introduction}
The  Erd\H os distance problem asks us to determine the minimal number of  distinct distances between any $N$ points in $\mathbb R^d.$
This problem was initially posed by Paul Erd\H os \cite{Er46} who conjectured
that $g_2(N) \gtrsim N/\sqrt{\log N}$ and $g_d(N) \gtrsim N^{2/d}$ for
$d \geq 3,$ where $g_d(N)$ denotes the minimal number of distinct distances between $N$ distinct points of $\mathbb{R}^d$, and  $X \gtrsim Y$ for $X,Y>0$ means that there is a constant $C>0$ independent of $N$ such that $ CX\ge Y.$
This conjecture in two dimension was resolved up to the logarithmic factor by Guth and Katz \cite{GuKa10}. However, the problem is still open in higher dimensions.\\

In \cite{BKT}, Bourgain, Katz, and Tao initially introduced the finite field analog of the Erd\H os distance problem and
Iosevich and Rudnev \cite{IR07} developed the problem in the general finite field setting. Let $\mathbb F_q^d$ be the $d$-dimensional vector space over the finite field $\mathbb F_q$ with $q$ elements. Throughout this paper we always assume that the characteristic of $\mathbb F_q$ is strictly greater than two. For a set $E\subset \mathbb F_q^d,$ the distance set, denoted by $\Delta_2(E)$, is defined as
$$\Delta_2(E)=\{ \|{\bf x}-{\bf y}\|\in \mathbb F_q: {\bf x}, {\bf y}\in E\},$$
where $\|\alpha\|=\alpha_1^2+ \cdots+ \alpha_d^2$ for $\alpha=(\alpha_1, \ldots, \alpha_d)\in \mathbb F_q^d.$
 With this definition of the distance set,
Bourgain, Katz, and Tao  \cite{BKT}  proved that  if  $q \equiv 3 \pmod{4}$ is prime and $E \subset \mathbb{F}_q^2$ with  $ q^{\delta} \lesssim |E| \lesssim q^{2 - \delta}$ for $\delta > 0,$ then there exists $\epsilon = \epsilon(\delta)$ such that $|\Delta_2(E)| \ge |E|^{1/2 + \epsilon}$.  
Here we recall that  $A\lesssim B$ for $A, B>0$ means that there exists a constant $C>0$ independent of $q$ such that $A\le CB,$ and we write $B\gtrsim A$ for $A\lesssim B.$ We also use $A\sim B$ if $\lim_{q\to \infty} A/B=1.$
This result was obtained by finding the connection between incidence geometry in $\mathbb{F}_q^2$ and the distance set.  Unfortunately, it is not simple to find the relationship between $\delta$ and $\epsilon$  from their proof. Furthermore, if $E=\mathbb F_q^2$, then  $|\Delta_2(E)|=\sqrt{|E|},$ which shows that the exponent $1/2$ can not be generically improved. If $-1$ is a square in $\mathbb F_q$, another unpleasant example exists with the finite field Erd\H os distance problem. For instance, let
$E=\{(t, it)\in \mathbb F_q^2: t\in \mathbb F_q\},$
where $i$ denotes an element of $\mathbb F_q$ such that $i^2=-1.$
Then it is straightforward to see that $|E|= q$ and $|\Delta_2(E)|=|\{0\}|=1.$
In view of aforementioned examples, Iosevich and Rudnev \cite{IR07} reformulated the Erd\H os distance problem in general finite field setting as follows.

\begin{question}\label{Q1.1} Let $E\subset \mathbb F_q^d.$ What is the smallest exponent $\beta>0$ such that
if $|E|\ge Cq^\beta$ for a sufficiently large constant $C>0$ then $|\Delta_2(E)|\ge c q$ for some $0<c\le 1?$
\end{question}
The problem in this question is called the Erd\H os-Falconer distance problem in the finite field setting.
Note that a distance can be viewed as an 1-dimensional simplex. Readers may refer to \cite{HI07, BIP14, V12, BHIPR17, PPV17} and references contained therein for the $k$-simplices problems. 
In \cite{IR07}, it was shown that $\beta\le (d+1)/2$ for all dimensions $d\ge 2.$ The authors in \cite{HIKR10} proved that $\beta=(d+1)/2$ for general odd dimensions $d\ge 3.$ On the other hand, they conjectured that if the dimension $d\ge 2$ is even, then $\beta$ can be improved to $d/2.$ In dimension two, the authors in \cite{CEHIK09} applied the sharp finite field restriction estimate for the circle on the plane so that they show $\beta\le 4/3$ which improves the exponent $(d+1)/2$, a sharp exponent in general odd dimensions $d.$ 
It was also observed in \cite{BHIPR17} that the exponent $4/3$ can be obtained by applying the group action.
Furthermore, considering the perpendicular bisector of two points in a set of $\mathbb F_q^2,$ the authors in \cite{HLR16} proved that the exponent $4/3$  holds for the pinned distance problem case. However, in higher even dimensions $d\ge 4,$ the exponent $(d+1)/2$ has not been improved. To demonstrate some possibility that the exponent $(d+1)/2$ could be improved for even dimensions $d\ge 4,$ the authors in \cite{CKP16} introduced a $k$-resultant modulus set  which generalizes the distance set in the sense that any $k$ points  can be selected from a set $E\subset \mathbb F_q^d$ to determine an object similar to a distance.
More precisely, for a set $E\subset \mathbb F_q^d$  we define a  $k$-resultant modulus set $\Delta_k(E)$ as
$$ \Delta_k(E)=\left\{\|{\bf x}^1\pm{\bf x}^2+\cdots\pm{\bf x}^k\|\in \mathbb F_q:{\bf x}^j\in E\right\}.$$
Since the sign $``\pm"$ does not affect on our results in this paper, we shall simply take $``+"$ signs.  That is, we will use the definition
$$ \Delta_k(E)=\left\{\| {\bf x}^1+{\bf x}^2+\cdots +{\bf x}^k\|\in \mathbb F_q:{\bf x}^j\in E\right\}$$
for consistency.
With this definition of the $k$-resultant modulus set, the following question was proposed in \cite{CKP16}.
\begin{question}\label{Q1.2} Let $E\subset \mathbb F_q^d$ and $k\ge 2$ be an integer.
What is the smallest exponent $\gamma>0$ such that if $|E|\ge C q^\gamma$ for a sufficiently large constant $C>0$, then $|\Delta_k(E)|\ge c q$ for some $0<c\le 1?$
\end{question}
This problem is called the $k$-resultant modulus problem.
When $k=2$, this question is simply the finite field  Erd\H os-Falconer distance problem, and in this sense the $k$-resultant modulus problem is a direct generalization of the distance problem. It is obvious that the smallest exponent $\beta>0$ in Question \ref{Q1.1} is greater than or equal to the smallest exponent $\gamma>0$ in Question \ref{Q1.2}. In \cite{CKP16}, it was  conjectured that $\gamma$ must be equal to $\beta.$
This conjecture means that the solution $\gamma$ of Question \ref{Q1.2} is independent of the integer $k\ge 2.$ In other words, it is conjectured that the solution of the  Erd\H os-Falconer distance problem is the same as that of the $k$-resultant modulus problem.
In fact, the authors in \cite{CKP17} provided a simple example which shows that $\gamma=\beta=(d+1)/2$ for any integer $k\ge 2$ in odd dimensions $d\ge 3$ provided that $-1$ is a square number in $\mathbb F_q.$  On the other hand,  they conjectured that the smallest exponent $\gamma$ in Question \ref{Q1.2} is $d/2$ for even dimensions $d\ge 2$ and  all integers $k\ge 2$. In addition, they showed that if $k\ge 3$ and the dimension $d$ is even, then  one can improve the exponent $(d+1)/2$ which is sharp in odd dimensional case.
More precisely they obtained the following result.
\begin{theorem}\label{Thm1.3} Let $E\subset \mathbb F_q^d.$ Suppose that $C$ is a sufficiently large constant. Then the following statements hold:\\
\noindent{\it (1)} If  $d=4$ or $6,$ and
 $|E|\geq C q^{\frac{d+1}{2}-\frac{1}{6d+2}},$   then $|\Delta_3(E)|\geq c q$ for some $0 < c \leq 1.$\\
\noindent{\it (2)} If $d\ge 8$ is even and $|E|\geq C q^{\frac{d+1}{2}-\frac{1}{6d+2}},$ then $|\Delta_4(E)|\geq cq$ for some $0 < c \leq 1.$\\
\noindent{\it (3)} If $d\geq 8$ is even,
then for any $\varepsilon>0,$ there exists $C_{\varepsilon}>0$ such that
if $|E|\ge C_{\varepsilon}  q^{\frac{d+1}{2} - \frac{1}{9d -18} + \varepsilon},$ then $|\Delta_3(E)|\geq cq$ for some $0 < c \leq 1.$
\end{theorem}

The purpose of this paper is to generalize Theorem \ref{Thm1.3}.
In particular, in the general setting, we improve the third conclusion of Theorem \ref{Thm1.3} by removing the $\varepsilon>0$ in the exponent.\\

The generalized Erd\H os-Falconer distance problem has been recently studied by considering the distances between any two sets $E_1, E_2\subset \mathbb F_q^d$ (see, for example, \cite{HLS10, KoS12, KS12, KS13, Sh06, V08, V13}).
Given two sets $E_1, E_2 \subset \mathbb F_q^d,$ the generalized  distance set $\Delta_2(E_1, E_2)$ is defined by
$$ \Delta_2(E_1, E_2)=\left\{ \|{\bf x}^1-{\bf x}^2\|\in \mathbb F_q: {\bf x}^1\in E_1, ~{\bf x}^2\in E_2\right\}.$$
The generalized Erd\H os-Falconer distance problem is to determine the smallest exponent $\gamma_2>0$ such that if $E_1, E_2\subset \mathbb F_q^d$ with $|E_1||E_2|\ge C q^{\gamma_2}$ for a sufficiently large constant $C>0$, then $|\Delta_2(E_1, E_2)|\ge c q$ for some $0<c\le 1.$
From the Erd\H os-Falconer distance conjecture, it is natural to conjecture that the smallest exponent $\gamma_2$ is $d+1$ for odd dimension $d\ge 3$ and $d$ for even dimension $d\ge 2.$ Shparlinski \cite{Sh06} obtained the exponent $d+1$ for all dimensions $d\ge 2.$ Thus, the generalized Erd\H os-Falconer distance conjecture was established in odd dimensions. On the other hand, the conjectured exponent $d$ in even dimensions has not been obtained. The currently best known result  is the exponent $d+1$ for even dimensions $d$ except for two dimensions. 
In dimension two the best known result is the exponent $8/3$ due to Koh and Shen \cite{KoS12}.\\

We now consider a problem which extends both the generalized Erd\H os-Falconer distance problem and the $k$-resultant modulus set problem. For $k$ sets $E_j\subset \mathbb F_q^d, j=1,2,\ldots,k,$ we define the generalized $k$-resultant set $\Delta_k(E_1, \ldots, E_k)$ as
$$ \Delta_k(E_1, \ldots, E_k)=\left\{\|{\bf x}^1+ {\bf x}^2+ \cdots + {\bf x}^k\|\in \mathbb F_q: {\bf x}^j\in E_j, ~j=1,2,\ldots,k \right\}.$$

\begin{problem}\label{Pro1.4} Let $k\ge 2$ be an integer. Suppose that  $E_j \subset\mathbb F_q^d, j=1,2,\ldots,k.$ Determine the smallest exponent $\gamma_k >0$ such that if $\prod\limits_{j=1}^k |E_j| \ge C q^{\gamma_k}$  for a sufficiently large constant $C>0$, then $|\Delta_k(E_1,E_2, \ldots, E_k)|\ge c q$ for some $0<c\le 1.$
\end{problem}
We call this problem the generalized $k$-resultant modulus problem.
As in the $k$-resultant modulus problem, we are only interested in studying this problem in even dimensions $d\ge 2.$ If $q=p^2$ for some odd prime $p$, then $\mathbb F_q$ contains the subfield $\mathbb F_p$. In this case, if the dimension $d$ is even and $E_j=\mathbb F_p^d$ for all $j=1,2,\ldots,k,$ then $\prod\limits_{j=1}^k|E_j|=q^{dk/2}$ and $|\Delta_k(E_1,E_2,\ldots, E_k)|=|\mathbb F_p|=p=\sqrt{q}.$ This example proposes the following conjecture.
\begin{conjecture} Suppose that $d\ge 2$ is even and $k\ge 2$ is an integer.
Then the smallest exponent $\gamma_k$ in Problem \ref{Pro1.4} must be $\frac{kd}{2}.$
\end{conjecture}

\subsection{Statement of the main result}
As mentioned before, the known results on the generalized Erd\H os-Falconer distance problem says that if $k=2,$ then  $2\le \gamma_k\le 8/3$ for $d=2$, and $d \le \gamma_k \le d+1 $ for even dimensions $d\ge 4,$ where $\gamma_k$ denotes the smallest exponent in Problem \ref{Pro1.4}. In this paper we study the generalized $k$-resultant modulus problem for $k\ge 3.$ Our main result is as follows.
\begin{theorem}\label{mainthm}
Let $k\ge 3$ be an integer and $E_j\subset \mathbb F_q^d$ for $j=1,2,\ldots,k. $ Assume that $C$ is a sufficiently large constant. Then the following statements hold:\\
\noindent{\it (1)} If  $d=4$ or $6,$ and
 $\prod\limits_{j=1}^3 |E_j| \ge C q^{3\left(\frac{d+1}{2} -\frac{1}{6d+2}\right)},$  then $|\Delta_3(E_1,E_2,E_3)|\gtrsim q.$\\
\noindent{\it (2)} If $d\ge 8$ is even and $\prod\limits_{j=1}^4 |E_j| \ge C q^{4\left(\frac{d+1}{2} -\frac{1}{6d+2}\right)},$ then $|\Delta_4(E_1,E_2,E_3, E_4)|\gtrsim q.$\\
\noindent{\it (3)} If $d\geq 8$ is even and $\prod\limits_{j=1}^3 |E_j| \ge C q^{3\left(\frac{d+1}{2} -\frac{1}{9d-18}\right)},$
then   $|\Delta_3(E_1,E_2,E_3)|\gtrsim q.$
\end{theorem}

Taking $E_j=E\subset \mathbb F_q^d$ for all $j=1,2,\ldots,k$, the first and second conclusions of Theorem  \ref{Thm1.3} follow immediately from $(1),(2)$ of Theorem \ref{mainthm}, respectively. Moreover, the third conclusion of Theorem \ref{mainthm} implies that the $\varepsilon$ in the statement $(3)$ of Theorem \ref{Thm1.3}  is not necessary. \\

Theorem \ref{mainthm} also implies the following result.
\begin{corollary} For any integer $k\ge 4,$ let $E_j\subset \mathbb F_q^d$ for $j=1,2,\ldots,k. $ Assume that  $C>0$ is a sufficiently large constant.
Then if $d\ge 4$ is even and  $\prod\limits_{j=1}^k |E_j| \ge C q^{k\left(\frac{d+1}{2} -\frac{1}{6d+2}\right)},$  we have $|\Delta_k(E_1,E_2,\ldots, E_k)|\gtrsim q.$
\end{corollary}
\begin{proof} Without loss of generality, we may assume that
$ |E_1|\ge |E_2|\ge \cdots \ge |E_k|.$
Notice that if $\prod\limits_{j=1}^k |E_j| \ge C q^{k\left(\frac{d+1}{2} -\frac{1}{6d+2}\right)},$ then  $\prod\limits_{j=1}^{k-1} |E_j| \gtrsim q^{(k-1)\left(\frac{d+1}{2} -\frac{1}{6d+2}\right)}.$ Since $|\Delta_{k-1}(E_1, \ldots, E_{k-1})| \le |\Delta_k(E_1, \ldots, E_k)|,$
the statement follows by induction argument with conclusions $(1),(2)$ of Theorem \ref{mainthm}.
\end{proof}

\section{Discrete Fourier analysis}
We shall use the discrete Fourier analysis to deduce the result of our main theorem, Theorem \ref{mainthm}. In this section, we recall notation and basic concept  in the discrete Fourier analysis. Throughout this paper, we shall denote by $\chi$ a nontrivial additive character of $\mathbb F_q.$
Since our result is independent of the choice of the character $\chi,$ we assume that $\chi$ is always a fixed nontrivial additive character of $\mathbb F_q.$  The orthogonality relation of $\chi$ states that
$$\sum_{\textbf{x}\in \mathbb F_q^d} \chi(\textbf{m}\cdot \textbf{x})=\left\{ \begin{array}{ll} 0  \quad&\mbox{if}~~ \textbf{m}\neq (0,\dots,0)\\
q^d  \quad &\mbox{if} ~~\textbf{m}= (0,\dots,0), \end{array}\right.$$
where  $\textbf{m}\cdot \textbf{x} :=\sum_{j=1}^d {\bf m}_j{\bf x}_j$ for ${\bf m}=({\bf m}_1, \ldots, {\bf m}_d), ~ {\bf x}=({\bf x}_1, \ldots, {\bf x}_d) \in \mathbb F_q^d.$  Given a function $g: \mathbb F_q^d \to \mathbb C,$ we shall denote by $\widetilde{g}$  the Fourier transform of $g$ which is defined by
\begin{equation}\label{gtile}
\widetilde{g}(\textbf{x})= \sum_{\textbf{m}\in \mathbb F_q^d} g(\textbf{m}) \,\chi(-\textbf{x}\cdot \textbf{m})\quad \mbox{for}~~\textbf{x}\in \mathbb F_q^d.
\end{equation}
On the other hand,  we shall denote by $\widehat{f}$ the {\bf normalized}  Fourier transform of the function $f: \mathbb F_q^d \to \mathbb C.$ Namely we define that
$$ \widehat{f}(\textbf{m})= \frac{1}{q^d} \sum_{\textbf{x}\in \mathbb F_q^d} f(\textbf{x}) \,\chi(-\textbf{x}\cdot \textbf{m})\quad \mbox{for}~~\textbf{m}\in \mathbb F_q^d.$$
In particular, if ${\bf m}=(0,\ldots,0)$ and we take
$f$ as an indicator function of a set $E\subset \mathbb F_q^d,$ then we see that
$$ \widehat{E}(0,\ldots,0)=\frac{|E|}{q^d}.$$
Here, and throughout this paper, we write $E({\bf x})$ for the indicator function $1_{E}({\bf x})$ of a set $E\subset \mathbb F_q^d.$
We define the {\bf normalized} inverse Fourier transform of $f$, denoted by  $ f^{\vee},$ as $f^{\vee}(\textbf{m})=\widehat{f}(-\textbf{m})$ for ${\bf m}\in \mathbb F_q^d.$ It is not hard to see that $\widetilde{ (f^{\vee})}(\textbf{x})=f(\textbf{x})$ for $\textbf{x}\in \mathbb F_q^d.$ Hence, we obtain the Fourier inversion theorem:
$$f(\textbf{x})= \sum_{\textbf{m}\in \mathbb F_q^d}  \widehat{f}(\textbf{m}) \,\chi(\textbf{m}\cdot \textbf{x}) \quad\mbox{for}~~\textbf{x}\in \mathbb F_q^d.$$
Using the orthogonality relation of the character $\chi,$ we see that
$$  \sum_{\textbf{m}\in \mathbb F_q^d} |\widehat{f}(\textbf{m})|^2 = \frac{1}{q^d} \sum_{\textbf{x}\in \mathbb F_q^d} |f(\textbf{x})|^2. $$
We shall call this formula the Plancherel theorem. Notice that if we take
$f$ as an indicator function of a set $E\subset \mathbb F_q^d,$ then the Plancherel theorem yields that
$$ \sum_{{\bf m}\in \mathbb F_q^d} |\widehat{E}({\bf m})|^2 = \frac{|E|}{q^d}.$$
Recall from H\"{o}lder's inequality that if $f_1, f_2: \mathbb F_q^d \to \mathbb C,$ then we have
$$ \sum_{{\bf m}\in \mathbb F_q^d} |f_1({\bf m})||f_2({\bf m})| \le \left( \sum_{{\bf m}\in \mathbb F_q^d} |f_1({\bf m})|^{p_1}\right)^{\frac{1}{p_1}} \left( \sum_{{\bf m}\in \mathbb F_q^d} |f_2({\bf m})|^{p_2}\right)^{\frac{1}{p_2}},$$
where $1< p_1, p_2 <\infty$ and $1/p_1 + 1/p_2=1.$ Applying  H\"{o}lder's inequality repeatedly, we see that if $f_i: D\subset \mathbb F_q^d \to \mathbb C$  and $1\le p_i <\infty$ for $i=1,2,\ldots, k$ with $\sum\limits_{i=1}^k \frac{1}{p_i} =1,$ then
\begin{equation}\label{GH} \sum_{{\bf m}\in D} \left(\prod_{i=1}^k |f_i({\bf m})|\right) \le \prod_{i=1}^k \left(\sum_{{\bf m}\in D}  |f_i({\bf m})|^{p_i}\right)^{\frac{1}{p_i}}.\end{equation}
We refer to this formula as the generalized H\"{o}lder's inequality.

\begin{lemma} \label{lem2.1} Let $k\ge 2$ be an integer. If $ E_j\subset \mathbb F_q^d$ for all $j=1,2,\ldots, k,$ then we have
$$ \sum_{{\bf m}\in \mathbb F_q^d} \left( \prod_{j=1}^k |\widehat{E_j}({\bf m})|\right)
\le q^{-dk+d} \left( \prod_{j=1}^k |E_j|\right)^{\frac{k-1}{k}}.$$
\end{lemma}
\begin{proof} Notice that if $E\subset \mathbb F_q^d$ and ${\bf m}\in \mathbb F_q^d,$ then $|\widehat{E}({\bf m})|\le |\widehat{E}(0,\ldots,0)|=\frac{|E|}{q^d} .$ Since $\sum\limits_{j=1}^k \frac{1}{k}=1,$ applying generalized H\"{o}lder's inequality yields the desirable result:
\begin{align*}
\sum_{{\bf m}\in \mathbb F_q^d} \left( \prod_{j=1}^k |\widehat{E_j}({\bf m})|\right)
&\le \prod_{j=1}^k \left(\sum_{{\bf m}\in \mathbb F_q^d} |\widehat{E_j}({\bf m})|^k \right)^{\frac{1}{k}}\\
&\le \prod_{j=1}^k \left(|\widehat{E_j}(0,\ldots,0)|^{\frac{k-2}{k}} \left(\sum_{{\bf m}\in \mathbb F_q^d} |\widehat{E_j}({\bf m})|^2\right)^{\frac{1}{k}}\right)\\
&=\prod_{j=1}^k \left( \left( \frac{|E_j|}{q^d}\right)^{\frac{k-2}{k}} \left(\frac{|E_j|}{q^d}\right)^{\frac{1}{k}}\right) =q^{-dk+d} \left( \prod_{j=1}^k |E_j|\right)^{\frac{k-1}{k}}
\end{align*}
\end{proof}

To estimate a lower bound of $|\Delta_k(E_1,\ldots, E_k)|$, we shall utilize the Fourier decay estimate on spheres. Recall that the sphere $S_t \subset \mathbb F_q^d$ for $t\in \mathbb F_q$ is defined by
\begin{equation} \label{defS_t} S_t=\{\textbf{x}\in \mathbb F_q^d: x_1^2+\cdots+x_d^2=t\}.\end{equation}
It is not hard to see that $|S_t|= q^{d-1}(1 + o(1))$ for $d\ge 3$ and $t\in \mathbb F_q$ (see Theorem 6.26 and Theorem 6.27 in \cite{LN97}).
It is well known that  the value of $\widehat{S_t}({\bf m})$ can be written in terms of the Gauss sum and the Kloosterman sum. In particular, when the dimension $d$ is even, the following result can be obtained from Lemma 4 in
\cite{IK10}.

\begin{lemma}\label{P2.1} Let $d\geq 2$ be even. If  $t\in \mathbb F_q$ and $\textbf{m}\in \mathbb F_q^d,$ then we have
$$  \widehat{S_t}(\textbf{m}) = q^{-1} \delta_0(\textbf{m}) + q^{-d-1} \,G^d \sum_{\ell \in {\mathbb F}_q^*}
\chi\Big(t\ell+ \frac{\|\textbf{m}\|}{4\ell}\Big),$$
where $\delta_0(\textbf{m})=1$ for $\textbf{m}=(0, \ldots, 0)$ and $\delta_0(\textbf{m})=0$ otherwise, and $G$ denotes the Gauss sum
$$
\displaystyle G =\sum_{s\in \mathbb F_q^*} \eta(s) \,\chi(s) ,
$$
where $\eta$ is the quadratic character of $\mathbb F_q,$ and $\mathbb F_q^* = \mathbb F_q \setminus \{0\}$.  In particular, we have
\begin{equation}\label{S_0}
\widehat{S_0}(\textbf{m})=q^{-1} \delta_0(\textbf{m}) + q^{-d-1}\, G^d \sum_{\ell \in {\mathbb F}_q^*} \chi( \|\textbf{m}\| \ell)\quad \mbox{for}~~\textbf{m}\in \mathbb F_q^d.\end{equation}
\end{lemma}
We shall invoke the following result which was given in Proposition 2.2 in \cite{KS13}.
\begin{lemma} \label{P2.2} If $\textbf{m}, \textbf{v} \in \mathbb F_q^d,$ then we have
$$  \sum_{t\in \mathbb F_q} \widehat{S_t}(\textbf{m})~ \overline{\widehat{S_t}}(\textbf{v}) =  q^{-1}\delta_0(\textbf{m})~\delta_0(\textbf{v}) + q^{-d-1} \sum_{s\in\mathbb F_q^*} \chi(s(\|\textbf{m}\|-\|\textbf{v}\|)).$$
\end{lemma}

\section{Formula for a lower bound of $|\Delta_k(E_1,\ldots,E_k)|$}
This section devotes to proving the following result which is useful to deduce a lower bound of $|\Delta_k(E_1,\ldots,E_k)|.$
\begin{theorem} \label{FormulaD} Let $d\ge 2$ be even and $k\ge 2$ be an integer.
If $ E_j\subset \mathbb F_q^d$ for $j=1,2,\ldots,k$ and $\prod\limits_{j=1}^k |E_j| \ge 3^k q^{\frac{dk}{2}}$, then we have
$$ |\Delta_{k}(E_1, \ldots, E_k)| \gtrsim
\min\left\{q, ~ \frac{\left( \prod\limits_{j=1}^k |E_j| \right)^{\frac{k+1}{k}}}
{ q^{dk}\, \left(\max\limits_{r\in \mathbb F_q^*} \prod\limits_{j=1}^k\left(\sum\limits_{{\bf v}\in S_r}  |\widehat{E_j}({\bf v})|^k \right)^{\frac{1}{k}}\right)}\right\}.$$
\end{theorem}
\begin{proof} For each $t\in \mathbb F_q,$ we define a counting function $\nu_k(t)$ by
$$ \nu_k(t)=\left|\{({\bf x}^1, {\bf x}^2,\ldots, {\bf x}^k)\in E_1\times \ldots \times E_k: \|{\bf x}^1+\cdots + {\bf x}^k\|=t\}\right|.$$
Since $\prod\limits_{j=1}^k |E_j| = \sum\limits_{t\in \mathbb F_q} \nu_k(t)$ and $\nu_k(t)=0$ for $t\notin \Delta_k(E_1,\ldots, E_k)$, we see that
$$ \prod_{j=1}^k |E_j| -\nu_k(0) =  \sum_{0\ne t\in \Delta_k(E_1,\ldots,E_k)} \nu_k(t).$$
 Square both sides of this equation and use the Cauchy-Schwarz inequality. It follows 
 $$ \left(\prod_{j=1}^k |E_j| -\nu_k(0)\right)^2=\left(\sum_{0\ne t\in \Delta_k(E_1,\ldots,E_k)}  \nu_k(t)\right)^2 \le |\Delta_k(E_1,\ldots,E_k)| \left(\sum_{t\in \mathbb F_q^*} \nu_k^2(t)\right).$$
 Thus, we obtain
 \begin{equation}\label{eq3.1} |\Delta_k(E_1,\ldots,E_k)|\ge \frac{\left(\prod\limits_{j=1}^k |E_j| -\nu_k(0)\right)^2}{\sum\limits_{t\in \mathbb F_q^*} \nu_k^2(t)}.\end{equation}
Now, we claim three facts below.
\begin{claim}\label{C1} Suppose that $d\ge 2$ is even and $k\ge 2$ is an integer.
If $E_j\subset \mathbb F_q^d$ for $j=1,2,\ldots,k$ with $\prod\limits_{j=1}^k |E_j| \ge 3^k q^{\frac{dk}{2}}$, then we have
$$ \left(\prod\limits_{j=1}^k |E_j| -\nu_k(0)\right)^2\ge \frac{1}{9} \left(\prod\limits_{j=1}^k |E_j|   \right)^2. $$
\end{claim}
\begin{claim} \label{C2} Let $d\ge 2$ and $k\ge 2$ be integers. If $E_j\subset \mathbb F_q^d$ for $j=1,2,\ldots, k,$ then we have
$$ \sum_{t\in \mathbb F_q} \nu_k^2(t) \le \frac{1}{q} \left( \prod_{j=1}^k |E_j|\right)^2 + q^{2dk-d} \sum_{r\in \mathbb F_q} \left| \sum_{{\bf v}\in S_r}\left(\prod_{j=1}^k \widehat{E_j}({\bf v})\right)\right|^2.$$
\end{claim}

\begin{claim}\label{C3} Assume that $d\ge 2$ is even and $k\ge 2$ is an integer.
If $E_j\subset \mathbb F_q^d$ for $j=1,2,\ldots,k$ with $\prod\limits_{j=1}^k |E_j| \ge q^{\frac{dk}{2}}$, then we have
$$ q^{2dk-d} \left| \sum_{{\bf m}\in S_0}\left(\prod_{j=1}^k \widehat{E_j}({\bf m})\right)\right|^2 -\nu^2_k(0) \le \frac{4}{q} \left( \prod_{j=1}^k |E_j|\right)^2.$$
\end{claim}
For a moment, let us accept Claims \ref{C1}, \ref{C2}, and \ref{C3} which shall be proved in the following subsections (see Subsections \ref{sub3.1}, \ref{sub3.2}, and \ref{sub3.3}). From \eqref{eq3.1} and Claim \ref{C1}, we see that if $\prod\limits_{j=1}^k|E_j| \ge 3^k q^{\frac{dk}{2}},$ then
 \begin{equation} \label{eq3.2}|\Delta_k(E_1,\ldots,E_k)|\gtrsim \frac{\left(\prod\limits_{j=1}^k |E_j|\right)^2}{\sum\limits_{t\in \mathbb F_q^*} \nu_k^2(t)}.\end{equation}
Observe from Claims \ref{C2} and \ref{C3} that if $\prod\limits_{j=1}^k|E_j| \ge 3^k q^{\frac{dk}{2}},$ then
\begin{align*} \sum_{t\in \mathbb F_q^*} \nu_k^2(t) &= \sum_{t\in \mathbb F_q} \nu_k^2(t)
-\nu_k^2(0) \le \frac{5}{q} \left( \prod_{j=1}^k |E_j|\right)^2  +q^{2dk-d} \sum_{r\ne 0} \left| \sum_{{\bf v}\in S_r}\left(\prod_{j=1}^k \widehat{E_j}({\bf v})\right)\right|^2\\
&\lesssim \frac{1}{q} \left( \prod_{j=1}^k |E_j|\right)^2
+q^{2dk-d} \left(\max_{r\ne 0} \sum_{{\bf v}\in S_r}\left(\prod_{j=1}^k |\widehat{E_j}({\bf v})|\right)\right) \left(\sum_{{\bf v}\in \mathbb F_q^d}    \left(\prod_{j=1}^k |\widehat{E_j}({\bf v})|\right) \right)\\
&\le \frac{1}{q} \left( \prod_{j=1}^k |E_j|\right)^2 +q^{dk} \left( \prod_{j=1}^k |E_j|\right)^{\frac{k-1}{k}}\left(\max_{r\ne 0} \sum_{{\bf v}\in S_r}\left(\prod_{j=1}^k |\widehat{E_j}({\bf v})|\right)\right), \end{align*}
where Lemma \ref{lem2.1} was used to obtain the last inequality.
From this estimate and \eqref{eq3.2}, it follows that
\begin{align*} |\Delta_k(E_1,\ldots,E_k)|&\gtrsim \frac{\left(\prod\limits_{j=1}^k |E_j|\right)^2}{\frac{1}{q} \left( \prod\limits_{j=1}^k |E_j|\right)^2 +q^{dk} \left( \prod\limits_{j=1}^k |E_j|\right)^{\frac{k-1}{k}}\left(\max\limits_{r\in \mathbb F_q^*} \sum\limits_{{\bf v}\in S_r}\left(\prod\limits_{j=1}^k |\widehat{E_j}({\bf v})|\right)\right)}\\
&\gtrsim  \min\left\{q, ~ \frac{\left( \prod\limits_{j=1}^k |E_j| \right)^{\frac{k+1}{k}}}
{ q^{dk}\, \left(\max\limits_{r\in \mathbb F_q^*} \sum\limits_{{\bf v}\in S_r} \left(\prod\limits_{j=1}^k |\widehat{E_j}({\bf v})|\right) \right)}\right\}.  \end{align*}
Then the statement of Theorem 3.1 follows by applying  generalized H\"{o}lder's inequality \eqref{GH}:
$$\max\limits_{r\in \mathbb F_q^*} \sum\limits_{{\bf v}\in S_r} \left(\prod\limits_{j=1}^k |\widehat{E_j}({\bf v})|\right)
\le\,  \max\limits_{r\in \mathbb F_q^*} \,\prod\limits_{j=1}^k\left(\sum\limits_{{\bf v}\in S_r}  |\widehat{E_j}({\bf v})|^k \right)^{\frac{1}{k}}.$$

\end{proof}

\subsection{Proof of Claim \ref{C1}} \label{sub3.1}
Suppose that $d\ge 2$ is even and $k\ge 2$ is an integer.
Let $E_j\subset \mathbb F_q^d$ for $j=1,2,\ldots,k$ with $\prod\limits_{j=1}^k |E_j| \ge 3^k q^{\frac{dk}{2}}.$ We aim to show that
\begin{equation} \label{eq3.3}\left(\prod\limits_{j=1}^k |E_j| -\nu_k(0)\right)^2\ge \frac{1}{9} \left(\prod\limits_{j=1}^k |E_j|   \right)^2.\end{equation}
To prove this, we begin by estimating the counting function $\nu_k(t)$ for $t\in \mathbb F_q.$ For each $t\in \mathbb F_q,$ it follows that
\begin{align*}
\nu_k(t)&=\left|\{({\bf x}^1, {\bf x}^2,\ldots, {\bf x}^k)\in E_1\times \ldots \times E_k: \|{\bf x}^1+\cdots + {\bf x}^k\|=t\}\right|\\
&=\sum_{({\bf x}^1, {\bf x}^2,\ldots,{\bf x}^k)\in E_1\times E_2 \times \cdots \times E_k} S_t({\bf x}^1+{\bf x}^2+ \cdots+ {\bf x}^k).
\end{align*}
Applying the Fourier inversion theorem to the indicate function $S_t({\bf x}^1+{\bf x}^2+ \cdots+ {\bf x}^k)$, it follows that
$$ \nu_k(t) = \sum_{({\bf x}^1, {\bf x}^2,\ldots,{\bf x}^k)\in \mathbb F_q^d\times \cdots \times \mathbb F_q^d} E_1({\bf x}^1) \cdots E_k({\bf x}^k) \sum_{{\bf m}\in \mathbb F_q^d} \widehat{S_t}({\bf m})\, \chi({\bf m}\cdot ({\bf x}^1+\cdots+{\bf x}^k)).$$
By the definition of the normalized Fourier transform, we can write
\begin{equation}\label{eq3.4}\nu_k(t)= q^{dk} \sum_{{\bf m}\in \mathbb F_q^d} \widehat{S_t}({\bf m}) \left( \prod_{j=1}^k \overline{\widehat{E_j}}({\bf m})\right). \end{equation}
To prove \eqref{eq3.3}, we first find an upper bound of $\nu_k(0)= q^{dk} \sum\limits_{{\bf m}\in \mathbb F_q^d} \widehat{S_0}({\bf m}) \left( \prod\limits_{j=1}^k \overline{\widehat{E_j}}({\bf m})\right).$ By \eqref{S_0} of Lemma \ref{P2.1}, we can write
\begin{align}\label{nonumber3.5} \nonumber \nu_k(0)&= q^{dk} \sum\limits_{{\bf m}\in \mathbb F_q^d} \left( \prod\limits_{j=1}^k \overline{\widehat{E_j}}({\bf m})\right)
\left(q^{-1} \delta_0(\textbf{m}) + q^{-d-1}\, G^d \sum_{\ell \in {\mathbb F}_q^*} \chi( \|\textbf{m}\| \ell)\right)\\
 & =q^{dk-1} \left(\prod_{j=1}^k \overline{\widehat{E_j}}(0,\ldots,0) \right)
+ q^{dk-d-1} \,G^d \sum_{{\bf m}\in \mathbb F_q^d}\left( \prod\limits_{j=1}^k \overline{\widehat{E_j}}({\bf m})\right)\sum_{\ell \in {\mathbb F}_q^*} \chi( \|\textbf{m}\| \ell)   \end{align}
Since $ \overline{\widehat{E_j}}(0,\ldots,0)=\frac{|E_j|}{q^d},\,
|G|=q^{1/2},$ and $|\sum_{\ell \in {\mathbb F}_q^*} \chi( \|\textbf{m}\| \ell)|\le q,$ it follows that
$$\nu_k(0)\le q^{-1} \prod_{j=1}^k |E_j| + q^{dk-d/2} \sum_{{\bf m}\in \mathbb F_q^d} \prod_{j=1}^k |\widehat{E_j}({\bf m})|$$

Applying  Lemma \ref{lem2.1},
$$\nu_k(0)\le q^{-1} \prod_{j=1}^k |E_j|+ q^{\frac{d}{2}} \left(\prod_{j=1}^k |E_j|   \right)^{\frac{k-1}{k}}.$$
Since $q\ge 3,$ it follows that
\begin{align*} \prod\limits_{j=1}^k |E_j| -\nu_k(0) &\ge \prod\limits_{j=1}^k |E_j| -q^{-1} \prod_{j=1}^k |E_j|- q^{\frac{d}{2}} \left(\prod_{j=1}^k |E_j|   \right)^{\frac{k-1}{k}}\\
&\ge \frac{1}{3}\prod\limits_{j=1}^k |E_j| + \left( \frac{1}{3} \prod\limits_{j=1}^k |E_j|- q^{\frac{d}{2}} \left(\prod_{j=1}^k |E_j|   \right)^{\frac{k-1}{k}}\right).\end{align*}
Note that if $\prod\limits_{j=1}^k |E_j|\ge 3^k q^{\frac{dk}{2}},$ then the second term above is nonnegative. Thus we obtain that
$$\prod\limits_{j=1}^k |E_j| -\nu_k(0) \ge  \frac{1}{3}\prod\limits_{j=1}^k |E_j|.$$
Squaring the both sizes, we complete the proof of Claim \ref{C1}.
\subsection{Proof of Claim \ref{C2}}\label{sub3.2}
We want to prove the following $L^2$ estimate of the counting function $\nu(t):$
\begin{equation}
 \sum_{t\in \mathbb F_q} \nu_k^2(t) \le \frac{1}{q} \left( \prod_{j=1}^k |E_j|\right)^2 + q^{2dk-d} \sum_{r\in \mathbb F_q} \left| \sum_{{\bf v}\in S_r}\left(\prod_{j=1}^k \widehat{E_j}({\bf v})\right)\right|^2.
\end{equation}
By \eqref{eq3.4}, we see that
\begin{align*}
\sum_{t\in \mathbb F_q} \nu_k^2(t) &= \sum_{t\in \mathbb F_q} \nu_k(t) \,\overline{\nu_k}(t)\\
&= q^{2dk} \sum_{{\bf m}, {\bf v}\in \mathbb F_q^d} \left( \prod_{j=1}^k \overline{\widehat{E_j}}({\bf m})\right) \left( \prod_{j=1}^k \widehat{E_j}{({\bf v})}\right) \sum_{t\in \mathbb F_q} \widehat{S_t}({\bf m}) \,\overline{\widehat{S_t}}({\bf v}).
\end{align*}

Using Lemma \ref{P2.2}, we see that
\begin{align*}
\sum_{t\in \mathbb F_q} \nu_k^2(t) =& \,q^{2dk-1} \sum_{{\bf m},{\bf v}\in \mathbb F_q^d} \left( \prod_{j=1}^k \overline{\widehat{E_j}}({\bf m})\right)\left( \prod_{j=1}^k \widehat{E_j}{({\bf v})}\right) \delta_0({\bf m})\,\delta_{0}({\bf v}) \\
&+ q^{2dk-d-1} \sum_{{\bf m}, {\bf v}\in \mathbb F_q^d} \left( \prod_{j=1}^k \overline{\widehat{E_j}}({\bf m})\right)\left( \prod_{j=1}^k \widehat{E_j}{({\bf v})}\right) \left( \sum_{s\in \mathbb F_q} \chi\left(s(\|{\bf m}\| - \|{\bf v}\|)\right) -1\right)\\
=& \,q^{2dk-1} \left(  \prod_{j=1}^k \overline{\widehat{E_j}}(0,\ldots,0)\right)\left( \prod_{j=1}^k \widehat{E_j}{(0, \ldots, 0)}\right) \\
&+ q^{2dk-d-1} \sum_{{\bf m}, {\bf v}\in \mathbb F_q^d} \left( \prod_{j=1}^k \overline{\widehat{E_j}}({\bf m})\right)\left( \prod_{j=1}^k \widehat{E_j}{({\bf v})}\right) \sum_{s\in \mathbb F_q} \chi\left(s(\|{\bf m}\| - \|{\bf v}\|)\right)\\
&- q^{2dk-d-1} \sum_{{\bf m}, {\bf v}\in \mathbb F_q^d} \left( \prod_{j=1}^k \overline{\widehat{E_j}}({\bf m})\right)\left( \prod_{j=1}^k \widehat{E_j}{({\bf v})}\right).
\end{align*}
By the definition of the normalized Fourier transform,  the orthogonality relation of $\chi$, and basic property of summation, it follows that
\begin{align*}
\sum_{t\in \mathbb F_q} \nu_k^2(t) =& \,q^{2dk-1} \left(\prod_{j=1}^k \frac{|E_j|}{q^d}\right)^2 \\
&+ q^{2dk-d} \sum_{{\bf m}, {\bf v}\in \mathbb F_q^d: \|{\bf m}\|=\|{\bf v}\|} \left( \prod_{j=1}^k \overline{\widehat{E_j}}({\bf m})\right)\left( \prod_{j=1}^k \widehat{E_j}{({\bf v})}\right)\\
&- q^{2dk-d-1} \left| \sum_{{\bf v}\in \mathbb F_q^d} \left(\prod_{j=1}^k \widehat{E_j}({\bf v})\right) \right|^2
\end{align*}
Since the third term above is not positive, we obtain that
\begin{align*}
\sum_{t\in \mathbb F_q} \nu_k^2(t) \le&\, q^{2dk-1} \left(\prod_{j=1}^k \frac{|E_j|}{q^d}\right)^2 + q^{2dk-d} \sum_{{\bf m}, {\bf v}\in \mathbb F_q^d: \|{\bf m}\|=\|{\bf v}\|} \left( \prod_{j=1}^k \overline{\widehat{E_j}}({\bf m})\right)\left( \prod_{j=1}^k \widehat{E_j}{({\bf v})}\right)\\
=& \,q^{-1} \left(\prod_{j=1}^k |E_j|\right)^2 + q^{2dk-d} \sum_{r\in \mathbb F_q} \left| \sum_{{\bf v}\in \mathbb F_q^d: \|{\bf v}\|=r} \left( \prod_{j=1}^k \widehat{E_j}({\bf v})\right)\right|^2,
\end{align*}
which completes the proof of Claim \ref{C2}.

\subsection{Proof of Claim \ref{C3}} \label{sub3.3}
For even $d\ge 2$ and an integer $k\ge 2$, let
 $E_j\subset \mathbb F_q^d$ for $j=1,2,\ldots,k$ with $\prod\limits_{j=1}^k |E_j| \ge q^{\frac{dk}{2}}.$ We must show that
\begin{equation}\label{eq3.7} q^{2dk-d} \left| \sum_{{\bf m}\in S_0}\left(\prod_{j=1}^k \widehat{E_j}({\bf m})\right)\right|^2 -\nu^2_k(0) \le \frac{4}{q} \left( \prod_{j=1}^k |E_j|\right)^2.\end{equation}
We begin by recalling from \eqref{nonumber3.5} that if $d\ge 2$ is even, then
$$ \nu_k(0)=q^{dk-1} \left(\prod_{j=1}^k \overline{\widehat{E_j}}(0,\ldots,0) \right)
+ q^{dk-d-1} \,G^d \sum_{{\bf m}\in \mathbb F_q^d}\left( \prod\limits_{j=1}^k \overline{\widehat{E_j}}({\bf m})\right)\sum_{\ell \in {\mathbb F}_q^*} \chi( \|\textbf{m}\| \ell).$$
It follows that
\begin{align*}
\nu_k(0) =& \,q^{-1} \left(\prod_{j=1}^k |E_j|\right) + q^{dk-d-1} \,G^d \sum_{{\bf m}\in \mathbb F_q^d}\left( \prod\limits_{j=1}^k \overline{\widehat{E_j}}({\bf m})\right)
\left(-1+\sum_{\ell \in {\mathbb F}_q} \chi( \|\textbf{m}\| \ell) \right)\\
=&\, \left[ q^{-1} \left(\prod_{j=1}^k |E_j|\right) - q^{dk-d-1} \,G^d \sum_{{\bf m}\in \mathbb F_q^d}\left( \prod\limits_{j=1}^k \overline{\widehat{E_j}}({\bf m})\right)\right]\\
 &+ q^{dk-d}\, G^d \sum_{{\bf m}\in S_0} \left( \prod_{j=1}^k \overline{\widehat{E_j}}({\bf m}) \right) := \mbox{A} + \mbox{B}.
\end{align*}
Thus we can write
\begin{align*}
\nu^2_k(0) &= \nu_k(0)\, \overline{\nu_k(0)} =(\mbox{A} + \mbox{B}) (\overline{\mbox{A}} + \overline{\mbox{B}})
= |\mbox{A}|^2 + |\mbox{B}|^2 + \mbox{A} \overline{\mbox{B}} + \overline{\mbox{A}} \mbox{B}
\end{align*}
Since the absolute value of the Gauss sum $G$ is $\sqrt{q}$, we have
$$\nu^2_k(0) = q^{2dk-d} \left|\sum_{{\bf m}\in S_0} \left(\prod_{j=1}^k \widehat{E_j}({\bf m})\right)\right|^2 + |\mbox{A}|^2+ \mbox{A} \overline{\mbox{B}} + \overline{\mbox{A}} \mbox{B}.$$
It follows that
\begin{equation}\label{eq3.8}
q^{2dk-d} \left| \sum_{{\bf m}\in S_0}\left(\prod_{j=1}^k \widehat{E_j}({\bf m})\right)\right|^2 -\nu^2_k(0) \le -\mbox{A} \overline{\mbox{B}} - \overline{\mbox{A}} \mbox{B} \le 2 |\mbox{A}| |\mbox{B}|.
\end{equation}
Now, notice that
$$ |A|\le q^{-1} \left(\prod_{j=1}^k |E_j|\right) + q^{dk-d-1} \,|G|^d \sum_{{\bf m}\in \mathbb F_q^d}\left( \prod\limits_{j=1}^k
|\widehat{E_j}({\bf m})|\right)$$
and
$$ |B|\le q^{dk-d}\,|G|^d \sum_{{\bf m}\in \mathbb F_q^d}\left(\prod_{j=1}^k |\widehat{E_j}({\bf m})|\right).$$
Since $|G|=\sqrt{q},$  using Lemma \ref{lem2.1} yields the following two estimates:
$$|A|\le q^{-1} \left(\prod_{j=1}^k |E_j|\right) +q^{\frac{d}{2}-1} \left(\prod_{j=1}^k |E_j|\right)^{\frac{k-1}{k}}$$
and
$$ |B| \le q^{\frac{d}{2}} \left(\prod_{j=1}^k |E_j|\right)^{\frac{k-1}{k}}.
$$
From these estimates and \eqref{eq3.8}, we have
\begin{align*} &q^{2dk-d} \left| \sum_{{\bf m}\in S_0}\left(\prod_{j=1}^k \widehat{E_j}({\bf m})\right)\right|^2 -\nu^2_k(0)\\
&\le 2 \left(q^{\frac{d}{2}-1}\left(\prod_{j=1}^k |E_j|\right)^{\frac{2k-1}{k}} +q^{d-1} \left(\prod_{j=1}^k |E_j|\right)^{\frac{2k-2}{k}}\right).\end{align*}
Finally, we obtain the estimate \eqref{eq3.7} by
observing that if $\prod\limits_{j=1}^k |E_j|\ge q^{\frac{dk}{2}}$, then
$$\max\left\{ q^{\frac{d}{2}-1}\left(\prod_{j=1}^k |E_j|\right)^{\frac{2k-1}{k}}~,~ q^{d-1} \left(\prod_{j=1}^k |E_j|\right)^{\frac{2k-2}{k}}\right\} \le q^{-1} \left(\prod_{j=1}^k |E_j|\right)^2.$$
Thus the proof of Claim \ref{C3} is complete.

\section{connection between restriction estimates for spheres and $|\Delta_k(E_1,\ldots,  E_k)|$}
 Theorem \ref{FormulaD} shows that  a good lower bound of $|\Delta_k(E_1,E_2,\ldots,E_k)|$ can be obtained by estimating an upper bound of the quantity
  \begin{equation}\label{kFourer}\max\limits_{r\in \mathbb F_q^*}\prod\limits_{j=1}^k\left(\sum\limits_{{\bf v}\in S_r}  |\widehat{E_j}({\bf v})|^k \right)^{\frac{1}{k}}.\end{equation}
 This quantity is closely related to the restriction estimates for spheres with non-zero radius. In this section, we review the restriction problem for spheres and we restate Theorem \ref{FormulaD} in terms of the restriction estimates for spheres.
We begin by reviewing the extension problem for spheres which is also called the dual restriction problem for spheres.  We shall use the notation $(\mathbb F_q^d, d{\bf x})$ to denote the $d$-dimensional vector space over the finite field $\mathbb F_q$ where a normalized counting measure $d{\bf x}$ is given. On the other hand, we denote by $(\mathbb F_q^d, d{\bf m})$ the dual space of the vector space $(\mathbb F_q^d, d{\bf x}),$ where we endow the dual space $(\mathbb F_q^d, d{\bf m})$ with the counting measure $d{\bf m}.$ Since the space $(\mathbb F_q^d, d{\bf x})$ can be identified with its dual space $(\mathbb F_q^d, d{\bf x})$ as an abstract group,
we shall use the notation $\mathbb F_q^d$ to indicate both the space and its dual space. To distinguish the space with its dual space, we always use the variable ${\bf x}$ for the element of the space $(\mathbb F_q^d, d{\bf x})$ with the normalized counting measure $d{\bf x}.$ On the other hand, the variable ${\bf m}$ will be used to denote the element of the dual space $(\mathbb F_q^d, d{\bf m})$ with the counting measure $d{\bf m}.$ For example, we write ${\bf x}\in \mathbb F_q^d$ and ${\bf m}\in \mathbb F_q^d$ for ${\bf x}\in (\mathbb F_q^d, d{\bf x})$ and ${\bf m}\in (\mathbb F_q^d, d{\bf m})$, respectively. With these notations, the classical norm notation can be used to indicate the following sums: for $1\le r<\infty,$
\begin{align*}
&\|g\|_{L^r(\mathbb F_q^d, d\textbf{m})}^r = \sum_{\textbf{m} \in \mathbb F_q^d} |g(\textbf{m})|^r,\\
& \|f\|_{L^r(\mathbb F_q^d, d\textbf{x})}^r = q^{-d}\,\sum_{\textbf{x} \in \mathbb F_q^d} |f(\textbf{x})|^r,
\end{align*}
and
\begin{equation*} \|g\|_{L^\infty(\mathbb F_q^d, d\textbf{m})} =\max_{\textbf{m} \in \mathbb F_q^d} |g(\textbf{m})|.
\end{equation*}
where $g$ is a function on $(\mathbb F_q^d, d{\bf m})$ and $f$ is a function on $(\mathbb F_q^d, d{\bf x}).$
For each $t\in \mathbb F_q^*,$ let $S_t\subset (\mathbb F_q^d, dx)$ be the sphere defined as in \eqref{defS_t}. We endow the sphere $S_t$ with the normalized surface measure $d\sigma$ which is defined by measuring the mass of each point on $S_t$ as ${1}/{|S_t|}.$ Notice that the total mass of $S_t$ is $1$ and we have
\begin{align}\label{normSdef}
&\| f\|_{L^r(S_t,  d\sigma)}^r=\frac{1}{|S_t|} \sum_{\textbf{x}\in S_t} |f(\textbf{x})|^r\quad \mbox{for}~~ 1\le r<\infty,\\
\nonumber& \| f\|_{L^\infty(S_t,  d\sigma)}= \max_{\textbf{x} \in S_t} |f(\textbf{x})|.\end{align}

We also recall that if $f:(S_t, d\sigma) \to \mathbb C,$ then the inverse Fourier transform of  $fd\sigma$ is defined by
$$
(fd\sigma)^{\vee}(\textbf{m})  =\frac{1}{|S_t|} \sum_{\textbf{x}\in S_t}f(\textbf{x}) \chi(\textbf{m}\cdot \textbf{x})\quad\mbox{for}~~\textbf{m}\in (\mathbb F_q^d, d\textbf{m}).
$$
Since the sphere $S_t$ is symmetric about the origin, we can write
\begin{equation*}\label{IFT} (d\sigma)^{\vee}(\textbf{m}) = \frac{q^d}{|S_t|} \widehat{S_t}(\textbf{m})\quad\mbox{for}~~\textbf{m}\in (\mathbb F_q^d, d\textbf{m}).\end{equation*}
With the above notation,  the extension problem for the sphere $S_t$ asks us to determine $1\leq p, r\leq \infty$ such that there exists  $C>0$ satisfying the following extension estimate:
\begin{equation}\label{m1}
\| (fd\sigma)^\vee \|_{L^r(\mathbb F_q^d, d\textbf{m})} \le C \|f\|_{L^p(S_t, d\sigma)} \quad\mbox{for all} ~~
f: S_t \to \mathbb C,
\end{equation}
where the constant $C>0$ may depend on $p, r, d, S_t,$ but it must be  independent of the functions $f$ and the size of the underlying finite field $\mathbb F_q.$ By duality, this extension estimate is the same as the following restriction estimate (see \cite{MT04,Ta04}) :
\begin{equation}\label{m2}
 \|\widetilde{g}\|_{L^{p^\prime}(S_t, d\sigma)}\leq C \|g\|_{L^{r^\prime}(\mathbb F_q^d, d\textbf{m})}
\quad\mbox{for all} ~~g: \mathbb F_q^d \to \mathbb C,
\end{equation}
where $\widetilde{g}$ is defined as in \eqref{gtile} and
$p^{\prime}$, $r^{\prime}$  denote the H\"{o}lder conjugates of $p$ and $r,$ respectively (namely, $1/p+1/p^{\prime}=1$ and $1/r+1/r^{\prime}=1).$\\

Now, we address the relation between the restriction estimates for spheres with non-zero radius  and a lower bound of $|\Delta_k(E_1, \ldots, E_k)|.$
By Theorem \ref{FormulaD} and the definition of the restriction estimates for spheres in \eqref{m2}, we obtain the following result.

\begin{lemma}\label{Formulalem} For even $d\ge 2$ and an integer $k\ge 2$, let $E_j \subset \mathbb F_q^d$ for $j=1,2,\ldots,k.$ Assume that   $\prod\limits_{j=1}^k |E_j| \ge 3^k q^{\frac{dk}{2}}$ and the following restriction estimate holds for some $1\le \ell < \infty$ and $\alpha \in \mathbb R$:
 \begin{equation}\label{E_jrestriction}\|\widetilde{E_j}\|_{L^{k}(S_r, d\sigma)}\lesssim q^\alpha  \|E_j\|_{L^{\ell}(\mathbb F_q^d, d\textbf{m})}
\quad\mbox{for all}~~ r\in \mathbb F_q^*,~~j=1,2,\ldots,k. \end{equation}
Then we have
$$ |\Delta_{k}(E_1, \ldots, E_k)| \gtrsim
\min\left\{q, ~ \frac{\left( \prod\limits_{j=1}^k |E_j| \right)^{\frac{k+1}{k}-\frac{1}{\ell}}} { q^{k\alpha+d-1}}\right\}.$$
\end{lemma}
\begin{proof}
By Theorem \ref{FormulaD}, it suffices to prove that
\begin{equation} \label{Su1}
q^{dk}\, \left(\max\limits_{r\in \mathbb F_q^*} \prod\limits_{j=1}^k\left(\sum\limits_{{\bf v}\in S_r}  |\widehat{E_j}({\bf v})|^k \right)^{\frac{1}{k}}\right) \lesssim q^{k\alpha+d-1} \left( \prod\limits_{j=1}^k |E_j| \right)^{\frac{1}{\ell}}.
\end{equation}
Since $\widehat{E_j}({\bf v}) =q^{-d}\widetilde{E_j}({\bf v})$ for $j=1,\ldots, k,$  we see that
$$ q^{dk}\, \left(\max\limits_{r\in \mathbb F_q^*} \prod\limits_{j=1}^k\left(\sum\limits_{{\bf v}\in S_r}  |\widehat{E_j}({\bf v})|^k \right)^{\frac{1}{k}}\right) = \left(\max\limits_{r\in \mathbb F_q^*} \prod\limits_{j=1}^k\left(\sum\limits_{{\bf v}\in S_r}  |\widetilde{E_j}({\bf v})|^k \right)^{\frac{1}{k}}\right).$$
 Using the definition of $\|\widetilde{E_j}\|_{L^{k}(S_r, d\sigma)}$ in \eqref{normSdef} and the fact that $|S_r|\sim q^{d-1},$  the above quantity is similar to the following value:
 $$ q^{d-1} \left(\max\limits_{r\in \mathbb F_q^*} \prod\limits_{j=1}^k \|\widetilde{E_j}\|_{L^k(S_r, d\sigma)} \right).$$
 By assumption \eqref{E_jrestriction}, this can be dominated by
 $$ q^{d-1} \left(\prod\limits_{j=1}^k \left(q^\alpha\|E_j\|_{L^{\ell}(\mathbb F_q^d, d\textbf{m})} \right)\right)= q^{k\alpha+d-1} \left( \prod\limits_{j=1}^k |E_j|^{\frac{1}{\ell}}\right) =q^{k\alpha+d-1} \left( \prod\limits_{j=1}^k |E_j| \right)^{\frac{1}{\ell}}.$$
 Putting all estimates together yields the inequality \eqref{Su1}, which completes the proof.
\end{proof}

\section{Restriction theorems for spheres}
We see from Lemma \ref{Formulalem} that  the restriction estimates for spheres play an important role in determining  lower bounds of the  cardinality of the generalized $k$-resultant set $\Delta_k(E_1, \ldots, E_k).$  In particular, our main result (Theorem \ref{mainthm}) will be proved by making an effort on finding possibly large exponent $\ell \ge 1$  such that the restriction inequality \eqref{E_jrestriction} holds for $k=3$ or $k=4.$
In this section, we shall obtain such restriction estimates.
To this end, we shall apply the following dual restriction estimate for spheres with non-zero radius due to the authors in \cite{IK10}.
\begin{lemma}[\cite{IK10}, Theorem 1]\label{exkoh}
If  $d\geq 4$ be even, then
\begin{equation}\label{R3.1}
 \|(Fd\sigma)^\vee\|_{L^4(\mathbb F_q^d, d\textbf{m})} \lesssim \|F\|_{L^{(12d-8)/{(9d-12)}}(S_t, d\sigma)} \quad \mbox{for all} ~~F\subset S_t , ~t\neq 0.
\end{equation}
\end{lemma}

To obtain a restriction estimate for spheres, we shall use the dual estimate of \eqref{R3.1}. To this end, it is useful to review Lorentz spaces in our setting. For a function $f: (S_t, d\sigma) \to \mathbb C,$ we denote by
$d_f$ the distribution function on $[0, \infty)$:
$$ d_f(a):= \frac{1}{|S_t|} \left|\{{\bf x}\in S_t: |f({\bf x})|>a\}\right|.$$
We see that for $1\le r\le \infty,$
$$\|f\|_{L^r(S_t, d\sigma)}^r= r\int_0^\infty s^{r-1} d_{f}(s)~ds.$$

The function $f^*$ is defined on $[0, \infty)$ by
$$ f^*(s):= \inf\{a>0: d_f(a)\le s\}.$$
For $1\le p,r\le \infty$ and a function $f:(S_t, d\sigma)\to \mathbb C$, define
$$\|f\|_{L^{p,r}(S_t, d\sigma)} :=\left\{\begin{array}{ll} \left(\int\limits_{0}^\infty \left( s^{1/p} f^*(s) \right)^r \frac{ds}{s}\right)^{1/r}\quad&\mbox{for}~~ 1\le r<\infty\\
\sup\limits_{s>0} s^{1/p} f^*(s) \quad &\mbox{for}~~ r=\infty. \end{array} \right.$$
In particular, we see that
$$ \|f\|_{L^{p,1}(S_t, d\sigma)}=\int_0^\infty s^{1/p-1} f^*(s) ~ds.$$
It is not hard to see that for $1\le p\le \infty$ and $1\le r_1\le r_2\le \infty$,
$$\|f\|_{L^{p,r_2}(S_t, d\sigma)} \lesssim \|f\|_{L^{p,r_1}(S_t, d\sigma)} \quad \mbox{and}\quad\|f\|_{L^{p,p}(S_t, d\sigma)} = \|f\|_{L^{p}(S_t, d\sigma)}.$$
See \cite{Gr04} for further information about Lorentz spaces.
With the above notation, the following fact can be deduced.
\begin{lemma} \label{Lp1}
Let $d\sigma$ be the normalized surface measure on the sphere $S_t\subset (\mathbb F_q^d, d{\bf x}).$ Assume that the estimate
\begin{equation}\label{ass5.2}\|(Fd\sigma)^\vee\|_{L^r(\mathbb F_q^d, d{\bf m})}\lesssim \|F\|_{L^p(S_t, d\sigma)}\end{equation}
holds for all subsets $F$ of $S_t.$ Then we have
$$ \|(fd\sigma)^\vee\|_{L^r(\mathbb F_q^d, d{\bf m})}\lesssim \|f\|_{L^{p,1}(S_t, d\sigma)} $$
for all functions $f: (S_t, d\sigma)\to \mathbb C.$
\end{lemma}
\begin{proof}
Without loss of generality, we may assume that $f$ is a nonnegative simple function given by the form
\begin{equation} \label{reductionf} f=\sum_{j=1}^N a_j 1_{F_j}\end{equation}
where $F_N \subset F_{N-1} \subset \cdots \subset F_2 \subset F_1$ and $a_j>0$ for all $j=1,2,\ldots, N.$ Notice that
$$f^*(s)=\sum_{j=1}^N a_j ~1_{\left[0, \frac{|F_j|}{|S_t|}\right]}(s).$$
It follows that
$$ \int_{0}^\infty s^{\frac{1}{p} -1} f^*(s)~ds
= \int_0^\infty s^{\frac{1}{p} -1} \sum_{j=1}^N a_j ~1_{\left[0, \frac{|F_j|}{|S_t|}\right]}(s)~ds
$$
$$= \sum_{j=1}^N a_j\int_0^{\frac{|F_j|}{|S_t|}} s^{\frac{1}{p} -1}~ds
=p\sum_{j=1}^N a_j \left(\frac{|F_j|}{|S_t|}\right)^{\frac{1}{p}} =p \sum_{j=1}^N a_j \|F_j\|_{L^p(S_t, d\sigma)}.$$
Namely, we see that
$$ \int_{0}^\infty s^{\frac{1}{p} -1} f^*(s)~ds \sim \sum_{j=1}^N a_j \|F_j\|_{L^p(S_t, d\sigma)}.$$
Using this estimate along with \eqref{reductionf} and the hypothesis \eqref{ass5.2},  we see that
\begin{align*} \|(fd\sigma)^\vee\|_{L^r(\mathbb F_q^d, d{\bf m})}&\le \sum_{j=1}^N a_j \,\|(F_jd\sigma)^\vee\|_{L^r(\mathbb F_q^d, d{\bf m})}\\
 &\lesssim \sum_{j=1}^N a_j \,\|F_j\|_{L^p(S_t, d\sigma)} \sim \int_{0}^\infty s^{\frac{1}{p} -1} f^*(s)~ ds =\|f\|_{L^{p,1}(S_t, d\sigma)}.\end{align*}
Hence, the proof is complete.\end{proof}

We shall invoke the following weak-type restriction estimate.
\begin{lemma} \label{Insook1} If $d\ge 4$ is even and we put $r_0=(12d-8)/(3d+4)$, then the weak-type restriction estimate
\begin{equation}\label{dres} \|\widetilde{g}\|_{L^{r_0, \infty}(S_t, d\sigma)}\lesssim \|g\|_{L^{\frac{4}{3}}(\mathbb F_q^d, d\textbf{m})}\end{equation}
holds for all $t\in \mathbb F_q^*$ and for all functions $g:(\mathbb F_q^d, d\textbf{m}) \to \mathbb C.$
\end{lemma}
\begin{proof}
Since $r_0=(12d-8)/(3d+4)$,  its dual exponent $r_0^\prime$ is given by
$$r_0^\prime=(12d-8)/(9d-12).$$
Combining Lemma \ref{exkoh} with Lemma \ref{Lp1}, it follows that
$$ \|(fd\sigma)^\vee\|_{L^4(\mathbb F_q^d, d{\bf m})}\lesssim \|f\|_{L^{r_0^\prime,1}(S_t, d\sigma)} $$
for all functions $f: (S_t, d\sigma)\to \mathbb C$ with $t\in \mathbb F_q^*.$
By duality, this estimate is same as \eqref{dres}, which completes the proof.
\end{proof}

The following restriction estimate will play an important role in proving the third part of Theorem \ref{mainthm}.
\begin{lemma} \label{minj1} If  $E\subset (\mathbb F_q^d, d{\bf m})$ and $|E|\ge q^{\frac{d-1}{2}},$ then we have
$$ \|\widetilde{E}\|_{L^2(S_t, d\sigma)} \lesssim \frac{|E|}{q^{\frac{d-1}{4}}} \quad\mbox{for all}~~ t\in \mathbb F_q^*.$$
\end{lemma}
\begin{proof} Since $\|\widetilde{E}\|^2_{L^2(S_t, d\sigma)}= \frac{1}{|S_t|} \sum\limits_{{\bf x}\in S_t} |\widetilde{E}({\bf x})|^2$ and $|S_t|\sim q^{d-1}$ for $t\in \mathbb F_q^*,$  it is enough to show that if $E\subset \mathbb F_q^d$ with $|E|\ge q^{\frac{d-1}{2}},$ then
\begin{equation}\label{resres}\sum_{\textbf{x} \in S_t} |\widetilde{E}(\textbf{x})|^2 \lesssim  q^{\frac{d-1}{2}} |E|^2. \end{equation}

 Notice from the definition of the Fourier transforms that
\begin{align*} \sum_{\textbf{x} \in S_t} |\widetilde{E}(\textbf{x})|^2 &= \sum_{\textbf{x} \in S_t} \sum_{ \textbf{m}, \textbf{m}' \in E} \chi(-\textbf{x} \cdot (\textbf{m} -\textbf{m}')) =\sum_{ \textbf{m}, \textbf{m}'\in E} q^d \widehat{S_t}(\textbf{m} -\textbf{m}')\\
&= q^d |E| \widehat{S_t}(0, \ldots, 0) + \sum_{\textbf{m}, \textbf{m}' \in E: \textbf{m}\ne \textbf{m}'} q^d \widehat{S_t}(\textbf{m} -\textbf{m}')\\
&\le |E||S_t| +  \left(\max_{\textbf{n} \in \mathbb F_q^d\setminus \{(0, \ldots, 0)\}} |\widehat{S_t}(\textbf{n})| \right) \sum_{\textbf{m}, \textbf{m}' \in E:\textbf{m}\ne \textbf{m}'} q^d\\
&\lesssim |E|q^{d-1} + |E|^2 q^d \left(\max_{\textbf{n} \in \mathbb F_q^d\setminus \{(0, \ldots, 0)\}} |\widehat{S_t}(\textbf{n})| \right).\end{align*}
Now, we apply the well known fact (Lemma 2.2 in \cite{IR07}) that  if $S_t \subset \mathbb F_q^d$ for $t\in \mathbb F_q^*$ and $d\ge 2,$ then
\[\left(\max_{\textbf{n} \in \mathbb F_q^d\setminus \{(0, \ldots, 0)\}} |\widehat{S_t}(\textbf{n})| \right) \lesssim q^{-\frac{d+1}{2}}.\]
Then we see that
\begin{equation*}\sum_{\textbf{x} \in S_t} |\widetilde{E}(\textbf{x})|^2 \lesssim |E|q^{d-1} + q^{\frac{d-1}{2}} |E|^2 \lesssim q^{\frac{d-1}{2}} |E|^2, \end{equation*}
 where the last inequality follows from our assumption that $|E|\ge q^{\frac{d-1}{2}}.$ Thus, \eqref{resres} holds and we complete the proof.

\end{proof}
Now, we introduce the interpolation theorem which enables us to derive the restriction estimates we need for the proof of our main results.
\begin{theorem}\label{RTI} Let $\Omega$ be a collection of subsets $E$ of $(\mathbb F_q^d, d{\bf m}).$ Assume that the following two restriction estimates hold for all sets $E\in\Omega$ and $1\le r_0< r_1\le \infty$:
$$ \|\widetilde{E}\|_{L^{r_0,\infty}(S_t, d\sigma)} \lesssim A_0(q,|E|):=A_0$$
and
$$ \|\widetilde{E}\|_{L^{r_1,\infty}(S_t, d\sigma)} \lesssim A_1(q,|E|):=A_1.$$
Then for $r_0<r<r_1,$ we have
 \begin{equation}\label{Ma1} \|\widetilde{E}\|_{L^r(S_t, d\sigma)} \lesssim  \left(\max\left\{\frac{2r}{r-r_0},~\frac{2r}{r_1-r}\right\}\right)^{\frac{1}{r}}
A_0^{\frac{r_0(r_1-r)}{r(r_1-r_0)}} A_1^{\frac{r_1(r-r_0)}{r(r_1-r_0)}}.\end{equation}
Namely, if $\frac{1}{r}=\frac{1-\theta}{r_0} + \frac{\theta}{r_1}$ for some $0<\theta<1,$ then we have
\begin{equation}\label{Ma2}\|\widetilde{E}\|_{L^r(S_t, d\sigma)} \lesssim  A_0^{1-\theta} ~A_1^\theta.\end{equation}
\end{theorem}

\begin{proof} Let  $\delta>0$ which will be chosen later.
$$\|\widetilde{E}\|_{L^r(S_t, d\sigma)}^r= r\int_0^\infty s^{r-1} d_{\widetilde{E}}(s)~ds $$
$$=r\int_0^\delta s^{r-1} d_{\widetilde{E}}(s)~ds
+ r\int_\delta^\infty s^{r-1} d_{\widetilde{E}}(s)~ds $$
$$=r\int_0^\delta s^{r-r_0-1} s^{r_0} d_{\widetilde{E}}(s)~ds
+ r\int_\delta^\infty s^{r-r_1-1} s^{r_1} d_{\widetilde{E}}(s)~ds $$
$$\le r \left(\sup_{0<s<\infty} s^{r_0} d_{\widetilde{E}}(s)\right)
\int_0^\delta s^{r-r_0-1} ~ds
+ r \left(\sup_{0<s<\infty} s^{r_1} d_{\widetilde{E}}(s)\right)
\int_\delta^\infty s^{r-r_1-1} ~ds $$
$$= \frac{r}{r-r_0} \|\widetilde{E}\|^{r_0}_{L^{r_0, \infty}(S_t, d\sigma)} ~\delta^{r-r_0} + \frac{r}{r_1-r} \|\widetilde{E}\|^{r_1}_{L^{r_1, \infty}(S_t, d\sigma)} ~\delta^{r-r_1}$$
$$\lesssim \max \left\{ \frac{r}{r-r_0},~\frac{r}{r_1-r}\right\}
\left(A_0^{r_0} ~\delta^{r-r_0} + A_1^{r_1} ~\delta^{r-r_1}\right).$$
Now we choose $\delta$ such that
$$ A_0^{r_0} ~\delta^{r-r_0}=A_1^{r_1} ~\delta^{r-r_1}.$$
Namely, we choose
$$\delta=A_0^{\frac{r_0}{r_0-r_1}} A_1^{\frac{r_1}{r_1-r_0}}.$$
It follows that
\begin{align*}\|\widetilde{E}\|_{L^r(S_t, d\sigma)}^r &\lesssim \max \left\{ \frac{2r}{r-r_0},~\frac{2r}{r_1-r}\right\}A_0^{r_0} ~\delta^{r-r_0}\\
&\lesssim \max \left\{ \frac{2r}{r-r_0},~\frac{2r}{r_1-r}\right\}A_0^{r_0}   A_0^{\frac{r_0(r-r_0)}{r_0-r_1}} A_1^{\frac{r_1(r-r_0)}{r_1-r_0}}\\
&=\max \left\{ \frac{2r}{r-r_0},~\frac{2r}{r_1-r}\right\} A_0^{\frac{r_0(r_1-r)}{r_1-r_0}} A_1^{\frac{r_1(r-r_0)}{r_1-r_0}},\end{align*}
which implies \eqref{Ma1}. By a direct computation, \eqref{Ma2} follows from \eqref{Ma1}.

\end{proof}

\section{Proof of main theorem (Theorem \ref{mainthm})}
In this section, we shall give the complete proof of Theorem \ref{mainthm}.
Since $\|\widetilde{g}\|_{L^{\infty, \infty}(S_t, d\sigma)} =\|\widetilde{g}\|_{L^{\infty}(S_t, d\sigma)} = \max\limits_{{\bf x}\in S_t} |\widetilde{g}({\bf x})|\le \|g\|_{L^1(\mathbb F_q^d, d{\bf m})},$ it is clear that
\begin{equation}\label{Intrivial}
 \|\widetilde{E}\|_{L^{\infty, \infty}(S_t, d\sigma)} \lesssim \|E\|_{L^1(\mathbb F_q^d, d{\bf m})}=|E| \quad\mbox{for all}~~ E \subset \mathbb F_q^d,~t\ne 0.
\end{equation}
On the other hand, it follows from Lemma \ref{Insook1} that if $d\ge 4$ is even, then
 \begin{equation}\label{In1}\|\widetilde{E}\|_{L^{\frac{12d-8}{3d+4}, \infty}(S_t, d\sigma)}\lesssim \|E\|_{L^{\frac{4}{3}}(\mathbb F_q^d, d\textbf{m})}=|E|^{\frac{3}{4}}\quad\mbox{for all}~~ E \subset \mathbb F_q^d,~t\ne 0.\end{equation}

 \subsection{Proof of statement $(1)$ of Theorem \ref{mainthm}}
We need the following lemma.
\begin{lemma} \label{F1}
 If $d=4$ or $6$, then we have
 $$ \|\widetilde{E}\|_{L^3(S_t, d\sigma)} \lesssim
 \|E\|_{L^{\frac{9d+12}{6d+14}}(\mathbb F_q^d, d{\bf m})} \quad\mbox{for all}~~ E \subset \mathbb F_q^d,~t\ne 0. $$ \end{lemma}
 \begin{proof}
 Note that if $d=4$ or $6,$ then $\frac{12d-8}{3d+4}<3<\infty.$
 Therefore, using Theorem \ref{RTI}, we are able to interpolate \eqref{Intrivial} and \eqref{In1} so that we obtain
 $$ \|\widetilde{E}\|_{L^3(S_t, d\sigma)} \lesssim |E|^{1-\theta}
 |E|^{\frac{3\theta}{4}} \quad\mbox{with}~~\theta=\frac{12d-8}{9d+12}.$$
 Namely, we see that
 $$\|\widetilde{E}\|_{L^3(S_t, d\sigma)} \lesssim |E|^{\frac{6d+14}{9d+12}}= \|E\|_{L^{\frac{9d+12}{6d+14}}(\mathbb F_q^d, d{\bf m})}.$$
 Thus, the proof is complete.
 \end{proof}
 We are ready to prove statement (1) of Theorem \ref{mainthm}.
 We aim to prove that if $d=4$ or $d=6$, and $ E_1,E_2,E_3\subset \mathbb F_q^d$ with $|E_1||E_2||E_3|\gtrsim q^{3\left(\frac{d+1}{2} -\frac{1}{6d+2}\right)}$, then $|\Delta_3(E_1,E_2,E_3)|\gtrsim q.$
 Combining Lemma \ref{Formulalem} and Lemma \ref{F1}, we see that
 $$ |\Delta_{3}(E_1, E_2, E_3)| \gtrsim
\min\left\{q, ~ \frac{\left( \prod\limits_{j=1}^3 |E_j| \right)^{\frac{4}{3}-\frac{6d+14}{9d+12}}} { q^{d-1}}\right\}.$$
Since $|E_1||E_2||E_3|\ge C q^{3\left(\frac{d+1}{2} -\frac{1}{6d+2}\right)}$ for a sufficiently large $C>0$,  we see from a direct computation that
$$ |\Delta_{3}(E_1, E_2, E_3)| \gtrsim q.$$

 \subsection{Proof of statement  $(2)$ of Theorem \ref{mainthm}}
 We shall utilize the following key lemma.
 \begin{lemma} \label{F2}If $d\ge 4$ is even, then we have
 $$ \|\widetilde{E}\|_{L^4(S_t, d\sigma)} \lesssim
 \|E\|_{L^{\frac{12d+16}{9d+18}}(\mathbb F_q^d, d{\bf m})} \quad\mbox{for all}~~ E \subset \mathbb F_q^d,~t\ne 0. $$
 \end{lemma}
 \begin{proof}
 Since $\frac{12d-8}{3d+4}<4<\infty$ for all even $d\ge 4,$ interpolating \eqref{Intrivial} and \eqref{In1} yields
 $$\|\widetilde{E}\|_{L^4(S_t, d\sigma)} \lesssim |E|^{1-\theta}
 |E|^{\frac{3\theta}{4}} \quad\mbox{with}~~\theta=\frac{3d-2}{3d+4}.$$
 Since $|E|^{1-\theta}  |E|^{\frac{3\theta}{4}} = |E|^{\frac{9d+18}{12d+16}}=\|E\|_{L^{\frac{12d+16}{9d+18}}(\mathbb F_q^d, d{\bf m})}$, the statement follows.
 \end{proof}
Let us prove statement (2) of Theorem \ref{mainthm}.
Recall that we must show that if $d\ge 8$ is even and $E_1,E_2,E_3,E_4\subset \mathbb F_q^d$ with $\prod\limits_{j=1}^4 |E_j| \gtrsim q^{4\left(\frac{d+1}{2} -\frac{1}{6d+2}\right)},$ then $|\Delta_4(E_1,E_2,E_3, E_4)|\gtrsim q.$
 Combining Lemma \ref{Formulalem} and Lemma \ref{F2}, we obtain that
 $$ |\Delta_{4}(E_1, E_2, E_3, E_4)| \gtrsim
\min\left\{q, ~ \frac{\left( \prod\limits_{j=1}^4 |E_j| \right)^{\frac{5}{4}-\frac{9d+18}{12d+16}}} { q^{d-1}}\right\}.$$
Since $|E_1||E_2||E_3||E_4|\ge C q^{4\left(\frac{d+1}{2} -\frac{1}{6d+2}\right)}$ for a sufficiently large $C>0$,  it follows  from a direct computation that
$$ |\Delta_{4}(E_1, E_2, E_3, E_4)| \gtrsim q.$$
 \subsection{Proof of statement  $(3)$ of Theorem \ref{mainthm}}
 We begin by proving the following lemma.
 \begin{lemma} \label{Kong}
 Let $E\subset \mathbb F_q^d.$ If $d\ge 8$ is even and $|E|\ge  q^{\frac{d-1}{2}},$ then we have
 $$ \|\widetilde{E}\|_{L^3(S_t, d\sigma)} \lesssim
 q^{\frac{-3d^2+23d-20}{36d-96}} \|E\|_{L^{ \frac{18d-48}{15d-46} }(\mathbb F_q^d, d{\bf m})}\quad \mbox{for all}~~t\in \mathbb F_q^*. $$
 \end{lemma}
 \begin{proof}

 Since $\|\widetilde{E}\|_{L^{2,\infty} (S_t, d\sigma)} \le \|\widetilde{E}\|_{L^{2} (S_t, d\sigma)}$ and $|E| \ge q^{\frac{d-1}{2}}$, we see from Lemma \ref{minj1} that
 \begin{equation} \label{hyb2}
 \|\widetilde{E}\|_{L^{2, \infty}(S_t, d\sigma)} \lesssim {q^{-\frac{(d-1)}{4}}} |E| \quad\mbox{for all}~~t\in \mathbb F_q^* \quad\mbox{and} ~E\subset \mathbb F_q^d ~~\mbox{with}~~|E|\ge  q^{\frac{d-1}{2}}.
 \end{equation}
As in \eqref{In1}, we also see that
 \begin{equation}\label{Hy1}\|\widetilde{E}\|_{L^{\frac{12d-8}{3d+4}, \infty}(S_t, d\sigma)}\lesssim |E|^{\frac{3}{4}}\quad\mbox{for all}~~ E \subset \mathbb F_q^d, ~t\in \mathbb F_q^*.\end{equation}

 Since $2<3< \frac{12d-8}{3d+4}$ for $d\ge 8,$  by using Theorem  \ref{RTI} we are able to interpolate \eqref{hyb2} and \eqref{Hy1}. Hence, if $d\ge 8$ is even and $|E|\ge q^{\frac{d-1}{2}}$, then  we have
 $$ \|\widetilde{E}\|_{L^3(S_t, d\sigma)} \lesssim \left({q^{-\frac{(d-1)}{4}}} |E|\right)^{1-\theta}
 |E|^{\frac{3\theta}{4}} \quad\mbox{with}~~\theta=\frac{6d-4}{9d-24}.$$
 By a direct computation, we conclude
 $$\|\widetilde{E}\|_{L^3(S_t, d\sigma)} \lesssim q^{\frac{-3d^2+23d-20}{36d-96}} |E|^{\frac{15d-46}{18d-48}} =q^{\frac{-3d^2+23d-20}{36d-96}} \|E\|_{L^{ \frac{18d-48}{15d-46}}(\mathbb F_q^d, d{\bf m})},$$
 which completes the proof of the lemma.
 \end{proof}
 Let us prove the statement  $(3)$ of Theorem \ref{mainthm} which states that
 if $d\geq 8$ is even and $\prod\limits_{j=1}^3 |E_j| \gtrsim q^{3\left(\frac{d+1}{2} -\frac{1}{9d-18}\right)},$
then   $|\Delta_3(E_1,E_2,E_3)|\gtrsim q.$
To prove this, let us first assume that one of $|E_1|, |E_2|, |E_3|$ is  less than $q^{\frac{d-1}{2}}$, say that $|E_3|< q^{\frac{d-1}{2}}.$
Then by our hypothesis that $|E_1||E_2||E_3| \gtrsim q^{3\left(\frac{d+1}{2} -\frac{1}{9d-18}\right)},$ it must follow that
 $$|E_1||E_2| \gtrsim q^{-\frac{(d-1)}{2}}  q^{3\left(\frac{d+1}{2} -\frac{1}{9d-18}\right)} = q^{(d+1)+ \frac{3d-7}{3d-6}} > q^{d+1}.$$
This implies that $|\Delta_2(E_1,E_2)| \gtrsim q,$ which was proved by Shparlinski \cite{Sh06}. Thus it is clear that $ |\Delta_3(E_1,E_2,E_3)|\gtrsim q$, because $|\Delta_3(E_1,E_2,E_3)|\ge |\Delta_2(E_1,E_2)|.$
For this reason, we may assume that all of $|E_1|, |E_2|, |E_3|$ are greater than or equal to $q^{\frac{d-1}{2}}$ and $|E_1||E_2||E_3|\gtrsim q^{3\left(\frac{d+1}{2} -\frac{1}{9d-18}\right)}.$
Combining Lemma \ref{Formulalem} with Lemma \ref{Kong} , we obtain that
 $$ |\Delta_{3}(E_1, E_2, E_3)| \gtrsim
\min\left\{q, ~ \frac{\left( \prod\limits_{j=1}^3 |E_j| \right)^{\frac{4}{3}-\frac{1}{\ell}}} { q^{3\alpha+d-1}}\right\},$$
where we take $\alpha=\frac{-3d^2+23d-20}{36d-96}$ and $\ell=\frac{18d-48}{15d-46}.$ By a direct comparison, it is not hard to see that
if $|E_1||E_2||E_3| \gtrsim q^{3\left(\frac{d+1}{2} -\frac{1}{9d-18}\right)} = q^{\frac{9d^2-9d-20}{6d-12}},$ then we have
$|\Delta_3(E_1,E_2,E_3)|\gtrsim q.$ We have finished the proof of the third part of Theorem \ref{mainthm}.

\end{document}